\documentclass[12pt,psamsfonts]{amsart}
\usepackage{amsmath,amsthm,amsfonts,amssymb}

\usepackage{eucal}
\newcommand{\1}{{1 \hspace{-0.35em} {\rm 1}}}
\newcommand{\I}{\mathcal{I}}

\newcommand{\lan}{\langle}
\newcommand{\ra}{\rangle}
\newcommand{\ve}{\varepsilon}
\newcommand{\T}{\mathcal{T}}

\newcommand{\F}{\mathcal{F}}

\newcommand{\Rep}{{\rm Rep}}

\newcommand{\g}{\mathfrak{g}}

\newcommand{\R}{\mathbb{R}}
\newcommand{\greq}{\stackrel{Gr}{=}}
\DeclareMathOperator{\FPdim}{FPdim}

\newcommand{\la}{{\lambda}}
\newcommand{\La}{{\Lambda}}
\newcommand{\al}{\alpha}

\DeclareMathOperator{\End}{End}

\DeclareMathOperator{\Hom}{Hom}

\newcommand{\C}{\mathbb C}
\newcommand{\CC}{\mathcal{C}}
\newcommand{\Z}{\mathbb Z}
\newcommand{\ot}{\otimes}

\newcommand{\be}{\mathbf{e}}
\numberwithin{equation}{section}

\newtheorem{theorem}[equation]{Theorem}

\newtheorem{lemma}[equation]{Lemma}

\newtheorem{prop}[equation]{Proposition}
\theoremstyle{definition}

\newtheorem{remark}[equation]{Remark}

\newtheorem{definition}[equation]{Definition}

\begin{document}

\title[Unitarizablity of premodular categories]
{Unitarizablity of premodular categories}

\author{Eric C.\ Rowell}
\address{Department of Mathematics,
    Texas A\&M University, College Station, TX 77843}
   \email{rowell@math.tamu.edu}
\date{\today}
\subjclass[2000]{Primary 17B37; Secondary 18D10, 20G42}

\begin{abstract}
We study the unitarizability of
 premodular categories constructed from representations of quantum group at roots of
unity.  We introduce \emph{Grothendieck unitarizability}
as a natural generalization of unitarizability to classes of premodular categories
with a common Grothendieck semiring.  We obtain new results for quantum groups of
Lie types $F_4$ and $G_2$, and improve the previously obtained results for Lie types $B$ and $C$.
\end{abstract}
\maketitle

\section{Introduction}
A unitary premodular category is a premodular category over $\C$ whose morphisms spaces have been
equipped with a positive-definite Hermitian form, compatible with the other structures (\emph{i.e.} braiding, duality etc.).  In particular for any
object $V$ the vector space $\End(V)$ become a Hilbert space, and the braiding morphisms act by unitary operators.  It is a difficult task, in general, to show that a given Hermitian premodular category is unitary (see \cite{Kir} and \cite{Wenzl1}).  However, one can study the weaker notion of \emph{pseudo-unitarity}.  A premodular category is pseudo-unitary if
the categorical dimension of any object is a positive number (although this is not the
most general definition, see Section \ref{background} below).  Pseudo-unitarity can often be checked directly, and since any unitary premodular category is automatically pseudo-unitary it is a useful statistic.
Furthermore, there do not appear to be any examples of pseudo-unitary Hermitian premodular categories that are not unitarizable.

Unitary premodular categories have important applications in several areas of mathematics
and physics.
Firstly, they are a rich source of unitary representations of the braid group.  Secondly, they can be used to construct $\rm{II}_1$ von Neumann factors (see \cite{Wenzl1}).  In case the categories are modular as well, one can construct 3-dimensional Topological Quantum Field Theories (TQFTs) (see \cite{Tur}).  They can also be regarded as algebraic models for exotic 2-dimensional physical (anyonic) systems \cite{RR} such as the quasi-particles in the Fractional Quantum Hall Effect.  Recently these last two applications have been combined to
form the basis for topological quantum computation \cite{FLKW}.  This has lead to renewed interest in constructing and studying unitary premodular categories.  Fortunately there
are several well-known constructions of premodular categories, one of which is the
subject of this article.

To any simple Lie algebra $\g$ and $q\in\C$ with $q^2$ a primitive
$\ell$th root of unity, one may associate a Hermitian premodular category $\CC(\g,q,\ell)$
over $\C$.
The construction is well-known (see for example \cite{BK}): one begins with the category
of finite-dimensional representations $\Rep(U_q\g)$ of Lusztig's integral form of the quantum
group $U_q\g$, then passes to the (ribbon) subcategory $\T$ of tilting modules (due to H. Andersen).  The quotient of $\T$ by the tensor ideal of negligible morphisms yielding the finite semisimple ribbon (=premodular) category $\CC(\g,q,\ell)$.  Kirillov Jr. \cite{Kir} defined an
Hermitian structure on $\CC(\g,q,\ell)$, and conjectured that for the choice $q=e^{\pi i/\ell}$ the form is positive definite provided $2\mid\ell$ if $\g$ is of Lie type $B, C$ or $F_4$ and $3\mid\ell$ if $\g$ is of Lie type $G_2$.  Subsequently, Wenzl \cite{Wenzl1} proved this conjecture, and we showed \cite{Rowell1} that the hypothesis $2\mid\ell$ is necessary for Lie types $B$ and $C$ for $\ell$ sufficiently large.  In fact in \cite{Rowell1} it is shown that unitarity fails in a much stronger sense: no Hermitian premodular category with the same Grothendieck semiring as $\CC(\mathfrak{so}_{2k+1},q,\ell)$ can even be pseudo-unitary when $\ell$ is odd with $4k+3\leq\ell $.  This result was obtained by appealing to classification results found in \cite{TbWz}.

In this paper we investigate several notions of unitarity for categories of the form $\CC(\g,q,\ell)$ where $\g$ is of Lie type $B, C, F_4$ and $\ell$ is odd or $\g$ is
of Lie type $G_2$ and $3\nmid\ell$.  The basic result is these cases is that if $\ell$
is large enough none of the categories $\CC(\g,q,\ell)$ is pseudo-unitary for any choice
of $q$.  We also show that for some small values of $\ell$, there \emph{are}
choices of $q$ so that $\CC(\g,q,\ell)$ is pseudo-unitary.  In addition we show that in
some cases (with $\ell$ small) there are unitary premodular categories with the same Grothendieck semiring as $\CC(\g,q,\ell)$.
However, these anomalies can be explained as ``low-rank coincidences."  This completes the story for these types of categories, complementing the results
in \cite{Wenzl1} and \cite{Rowell1}, sharpening the results of the latter.

Here is a more detailed description of the contents of this paper.  In Section \ref{background}
we briefly describe the properties of premodular categories that are germane to our problem, introduce the new notion of Grothendieck unitarizability.  and describe
the premodular categories coming from quantum groups at roots of unity
In Section \ref{results} we prove some general
results for these categories and then prove our main results for Lie types $G_2$, $F_4$, $B$ and $C$.  In the Appendix we give an explicit equivalence between two quantum group categories
of Lie types $F_4$ and $E_8$ respectively.
\subsection*{Acknowledgments}
The author would like to thank an anonymous referee for an especially careful reading of
the manuscript and for providing a simple proof extending the known bound of $4k+3\leq\ell$ to
the sharp bound of 
$2k+5\leq\ell$ in Theorem \ref{bcthm}, originally stated as a conjecture.
\section{Background}\label{background}

Here we establish the setting in which we choose to work.  For the most part
we will be interested in Hermitian premodular categories.
By a \textit{premodular category} $\CC$ we mean an abelian semisimple $\C$-linear ribbon
category with finitely many isomorphism classes of simple objects (see \cite{Tur} for
a more detailed description). In other settings a premodular
category is called a \textit{ribbon fusion category}. A \textit{Hermitian} premodular category is equipped with an involution $f\mapsto\overline{f}$ on morphisms so that
the pairing $(f,g)\rightarrow tr_\CC(f\overline{g})$ is Hermitian and non-degenerate.
One also requires the involution to restrict to ordinary complex conjugation on the base
field $\C$.

\subsection{Grothendieck Equivalence}
For any semisimple braided tensor category $\CC$ the tensor product decomposition rules are
described by the \textbf{Grothendieck ring} $Gr(\CC)$: unital based ring (see \cite{Ostrik1})  with basis consisting of the isomorphism classes of simple objects $[X]$ with
$[X]\cdot[Y]=[X\ot Y]$, and $[X]+[Y]=[X\oplus Y]$, with unit $[\1]$ where $\1$ is the unit object.  We will sometimes abuse notation and write $X$ instead of $[X]$ for the elements of $Gr(\CC)$
as this will cause no confusion.  If there are finitely many simple classes
$[X_i]$, $0\leq i\leq n-1$ we say $\CC$ has \emph{rank} $n$, and $Gr(\CC)$ can be identified
with the finite rank unital based ring $\Z[x_0,x_1,\ldots,x_{n-1}]/\I$ where the ideal $\I$ is generated by
$\{x_ix_j-\sum N_{i,j}^kx_k\}$ given that $X_i\ot X_j\cong\sum_k N_{i,j}^kX_k$.  The structure constants $N_{i,j}^k$ satisfy a number of useful symmetries.  If we denote
by $i^*$ the index corresponding to the object $X_i^*$ dual to $X_i$ the we have:
\begin{eqnarray}
N_{i,j}^k=N_{j,i}^k=N_{i,k^*}^{j^*}=N_{i^*,j^*}^{k^*} \quad \text{and}\quad N_{i,j}^0=\delta_{j,j^*}
\end{eqnarray}
The Grothendieck ring has a useful matrix representation given by $X_i\rightarrow N_i$ where
$N_i$ is the $n\times n$ matrix with $(k,j)$ entry equal to $N_{i,j}^k$.  By the above symmetries one sees that $N_{i^*}=N_i^T$ where ${}^T$ is the transpose operator.  If $i=i^*$
for all $i$ (\emph{i.e.} $\CC$ is a \emph{self-dual} category), then the $N_{i,j}^k$ are totally symmetric in their indices.

The Grothendieck ring plays an important role in the theory of braided tensor categories due to a result that
is known as \emph{generalized Ocneanu rigidity} (see \cite{ENO}[Section 2.7]):
\begin{prop}\label{ocneanu}
For a fixed unital based ring $A$, there are at most finitely many inequivalent
premodular categories $\CC$ with $Gr(\CC)\cong A$.
\end{prop}
We should point out that this theorem is stated in \cite{ENO} for braided fusion categories,
but any such category has at most finitely many ribbon structures (see, for example, \cite{Kas}[Lemma XIV.3.4]).

Motivated by this theorem we make the following definition:
\begin{definition}
Let $\CC$ and $\F$ be two premodular categories.  We say $\CC$ and $\F$ are
\textbf{Grothendieck equivalent}, and write $\CC\greq\F$ if $Gr(\CC)\cong Gr(\F)$ as
unital based rings.
\end{definition}
An isomorphism between unital based rings is a bijection between the bases compatible with the fusion rules, so to show that two
categories are Grothendieck equivalent it is enough to show that there is a bijection
between the given bases that preserves the structure constants $N_{ij}^k$.

We have the following generalized notion of dimension:
\begin{definition}
A \textbf{dimension function} for a braided fusion category $\CC$ is a ring homomorphism
$\phi:Gr(\CC)\rightarrow \C$.
\end{definition}
In any ribbon category $\CC$ there is a canonical dimension function $\dim_\CC$ defined
for any object $X$ as the categorical trace of the identity morphism $Id_X\in\End(X)$.
It is well-known that $\dim_\CC(X)=\dim_\CC(X^*)$ for any object $X$, and that $\dim_\CC(X)\in\R$.  Notice that if $\phi$ is a dimension function for $\CC$ then the vector $\mathbf{f}=\sum_i\phi(X_i)\be_i$ is an eigenvector for each matrix $N_i$ with eigenvalue $\phi(X_{i^*})$ (where $\be_i$ is a standard basis vector for $\R^n$).
This follows from the fact that $N_{i}\be_j=\sum_k N_{i,j}^k\be_k$ and $\phi(X_{i^*})\phi(X_k)=\sum_j N_{i^*,k}^j\phi(X_j)$ and the calculation:
$$N_i\mathbf{f}=\sum_j\phi(X_j)\sum_kN_{i,j}^k\be_k=\sum_k(\sum_j N_{i^*,k}^j\phi(X_j))\be_k=\sum_k\phi(X_{i^*})\phi(X_k)\be_k=\phi(X_{i^*})\mathbf{f}.$$

\begin{definition}
If $\CC$ is a ribbon category, the \textbf{global dimension} $\dim(\CC)$ of the category
 is the sum of the squares of the categorical dimensions of the simple objects.
\end{definition}
 In more
general settings the global dimension is defined differently (see \cite{ENO}), but for ribbon (or more generally, spherical) categories
the definition we give is equivalent to the standard one.

\begin{definition}
The \textbf{Frobenius-Perron dimension} of an object $X$ in $\CC$ is the largest real eigenvalue
of the matrix $N_X$ of multiplication by $X$ in the the Grothendieck ring $Gr(\CC)$, and
is denoted $\FPdim(X)$.  $\FPdim(\CC)$ is defined to be the sum of the squares of $\FPdim(X_i)$ for simple classes $X_i$.  A category is called \textbf{pseudo-unitary}
if $\dim(\CC)=\FPdim(\CC)$.
\end{definition}

From \cite{ENO}
we extract some important properties of $\FPdim$:
\begin{enumerate}
\item $\FPdim(X)>0$ for any object (in particular $\FPdim(X)\in\R$).
\item $X\rightarrow\FPdim(X)$ defines a dimension function.
\item $\FPdim(X)$ is the unique dimension function with $\FPdim(X)>0$ for all objects $X$.
\end{enumerate}
It also follows from the argument above that $\FPdim(X)=\FPdim(X^*)$, since the largest
eigenvalues of $M$ and $M^T$ are the same.

Pseudo-unitarity does not immediately imply that $\dim(X_i)=\FPdim(X_i)$ for each simple object.
It is true (see \cite{ENO}[Prop. 8.3]) that if $\CC$ is any pseudo-unitary fusion category
then there is a unique \textit{spherical} structure on $\CC$ so that $\dim$ and $\FPdim$ coincide, but if $\CC$ is a pseudo-unitary ribbon category there may not be a \emph{ribbon} structure on $\CC$
satisfying $\dim(X_i)=\FPdim(X_i)$ for all simple $X_i$.  In general we have the following:
\begin{lemma}\label{dimlemma}
Suppose $\CC$ is a pseudo-unitary premodular category.  Then
\begin{enumerate}
\item[(a)] $\dim(X_i)=\pm\FPdim(X_i)$ and
\item[(b)] if $X_i$ is a subobject of $X_j\ot X_j^*$ for some simple $X_j$ then $\dim(X_i)=\FPdim(X_i)$.
\end{enumerate}
\end{lemma}
\begin{proof}
Part (a) follows from
$$\dim(\CC)=\sum_i(\dim(X_i))^2=\sum_i(\FPdim(X_i))^2=\FPdim(\CC)$$
and
 $|\dim(X_i)|\leq \FPdim(X_i)$ (by the Perron-Frobenius theorem).  If $\dim(X_i)<0$ for some
simple $X_i$ then $\CC_0:=\{X_i:\dim(X_i)=\FPdim(X_i)\}$ and $\CC_1:=\{X_i:\dim(X_i)=-\FPdim(X_i)\}$ are (non-empty) bases for a $\Z_2$-grading on $Gr(\CC)$.
Since both $X_i\rightarrow\dim(X_i)$ and $X_i\rightarrow\FPdim(X_i)$ give rise to characters
of $Gr(\CC)$ it is clear that this grading is well-defined.  Any simple subobject of $X_i\ot X_i^*$ is in $\CC_0$ since $\dim(X_i)=\dim(X_i^*)$, giving (b).
\end{proof}

Finally, we give the following generalization of unitarizability:
\begin{definition}
Let $\CC$ be any premodular category, and $[\CC]$ its Grothendieck equivalence class.
We say $\CC$ is \textbf{Grothendieck unitarizable} if there is a unitary premodular category $\F$ with $\F\in[\CC]$.
\end{definition}
\begin{remark}
The original form of Ocneanu rigidity states that there are finitely many \textit{fusion} categories sharing a common Grothendieck semiring.  So all of the notions described here work equally well for fusion categories.  However we are interested in fusion categories that are at least braided so we have made our definitions in this more limited setting.
\end{remark}

\subsection{Quantum groups at roots of unity}

In order to define and work with the categories $\CC(\g,q,\ell)$, we will need some
(standard) notation and definitions.  We make no attempt to be entirely self-contained, referring the
reader to the survey paper \cite{Survey} and the texts \cite{BK} and \cite{Tur} for full details.

Let $\g$ be a simple Lie algebra, $\Phi$ its root system
(resp. coroot system), with root basis $\Pi$  and
positive roots $\Phi_+$.  We denote by $\check{\Phi},\check{\Pi}$ and $\check{\Phi}_+$ the
corresponding sets of coroots.
Let $\lan\; ,\,\ra$ be the symmetric bilinear form on $\R\Phi$ normalized so that
$\lan\al,\al\ra=2$ for \emph{short} roots.  We embed $\check{\Phi}$ and $\Phi$ in the same vector space by means of the identification $\check{\al}=\frac{2\al}{\lan\al,\al\ra}$.  Observe that if $\al$ is a short root, $\check{\al}=\al$ under this identification.  Denote by $\theta$ the highest root and by
$\theta_s$ the highest short root.  Define $m=\frac{\lan\beta,\beta\ra}{\lan\al,\al\ra}$ where
$\beta$ is a long root and $\al$ is a short root, so that $m=1$ for Lie types $A,D$ and $E_N$ ($N=6,7$ or $8$),
$m=2$ for Lie types $B,C$ and $F_4$ and $m=3$ for Lie type $G_2$.  Denote by $\rho$ the
half-sum of the positive roots, $P$ the set of weights, and $P_+$ the set of dominant weights
labeling the irreducible highest-weight representations of $\g$.

The \textbf{Weyl group} of $\g$ is the (finite) group $W$ generated by the $|\Pi|$ \textit{fundamental reflections} defined on $\R\Phi$ for each $\al_i\in\Pi$ by
$s_i(\la)=\la-\lan\la,\check{\al}_i\ra\al_i$.  The \textbf{affine Weyl group} $W_\ell$ is generated by the fundamental reflections $s_i$ and an additional translation operator $T_\ell$.  If $m\mid\ell$, then $T_\ell$ is the translation by $\ell\check{\theta}$ while if
$m\nmid\ell$ $T_\ell$ is the translation by $\ell\check{\theta_s}=\ell\theta_s$.  Often we
will need to use the ``dot" action of the (affine) Weyl group defined by $w.\la=w(\la+\rho)-\rho$.  There is a well-defined multiplicative sign function $\ve$ on $W_\ell$ given by $\ve(s_i)=-1$, $\ve(T_\ell)=1$ and extended to products.

With these definitions we can describe the Hermitian ribbon category $\CC(\g,q,\ell)$.  The simple objects $X_\la$
are labeled by the intersection $C_\ell$ of the set of dominant weights $P_+$ with the fundamental domain of the dot action of $W_\ell$ containing $0$.  Explicitly, we have: $$C_\ell:=\begin{cases}\{\la\in P_+: \lan\la+\rho,\theta\ra<\ell\} & \text{if}\quad m\mid\ell\\
\{\la\in P_+: \lan\la+\rho,\theta_s\ra<\ell\} & \text{if}\quad m\nmid\ell\end{cases}$$

The tensor
product decomposition rules for $\CC(\g,q,\ell)$ are $W_\ell$-antisymetrizations of
the tensor product decomposition rules for $\Rep(U_q\g)$ with $q$ not a root of unity or, equivalently, of $\Rep(\g)$.  Define
$m_{\la\mu}^\nu=\dim\Hom_\g(V_\nu,V_\la\ot V_\mu)$, \emph{i.e.} the multiplicity of the irreducible Weyl
module $V_\nu$ (over $\g$) in the decomposition of $V_\la\ot V_\mu$ into irreducible Weyl modules.  Similarly, denote by
$N_{\la\mu}^\nu$ the multiplicity of the object $X_\nu$ in $X_\la\ot X_\mu$.  Then we have the
formula (see \cite{AndPar} and \cite{Sawin}):

\begin{equation}\label{qRacah}
N_{\la\mu}^\nu=\sum_{\stackrel{w\in W_\ell}{w.\nu\in P_+}} \ve(w)m_{\la\mu}^{w.\nu}
\end{equation}

Having chosen some ordering on the set $C_\ell$ the matrices $N_\la$ give a faithful representation of $Gr(\CC(\g,q,\ell))$ by
extending $X_\la\rightarrow N_\la$ to all objects.

\begin{remark}
Observe that the structure constants $N_{\la\mu}^\nu$ depend only on $\ell$, not on the particular choice of $q$ with $q^2$ a
primitive $\ell$th root of unity.  This implies that the class of categories $\mathfrak{C}(\g,\ell):=\{\CC(\g,q,\ell): q=e^{z\pi i/\ell}, \gcd(z,\ell)\}$
are Grothendieck equivalent.  It follows from the strong form of Ocneanu rigidity (Proposition \ref{ocneanu}) that there are finitely
many inequivalent premodular categories in $\mathfrak{C}(\g,\ell)$.
\end{remark}

The dimension function for $\CC(\g,q,\ell)$ is computed (for simple objects) by the formula:
\begin{equation}\label{qdim}
\dim_q(X_\la)=\prod_{\al\in\Phi_+}\frac{[\lan\la+\rho,\al\ra]}{[\lan\rho,\al\ra]}
 \end{equation}
where $[n]=\frac{q^n-q^{-n}}{q-q^{-1}}$ is the
usual $q$-number.  The categorical dimension is obtained by linear
extension of this formula to direct sums of simple objects.

\begin{prop}[Wenzl]  Suppose $m\mid\ell$, and set $q=e^{\pi i/\ell}$.  Then
$\dim_q$ is the Frobenius-Perron dimension for
for all of the Grothendieck equivalent categories in $\mathfrak{C}(\g,\ell)$.
\end{prop}

\section{Main results}\label{results}
We can now give the main results of the paper.  We first prove some
general results and then apply them to categories of Lie types $G_2$, $F_4$, $B$ and
$C$.  In what follows the categories of Lie types $G_2$ and $F_4$ are always modular, while
the categories of type $B$ and $C$ may be only premodular.

\subsection{General results}

In \cite{Rowell1} the $\FPdim$ function was determined for categories $\CC(\mathfrak{so}_{2k+1},q,\ell)$, with $\ell$ odd.
  The proof found there contained some Lie-type-$B$-specific \emph{ad-hoc} arguments that do not extend to
other Lie types.  The following generalizes that result:
\begin{theorem}\label{fpthm}
Assume that $m\nmid\ell$ or $m=1$, and define
\begin{equation}\label{fpdim}
d_\la(q):=\prod_{\check{\al}\in\check{\Phi}_+}\frac{[\lan\la+\rho,\check{\al}\ra]}{[\lan\rho,\check{\al}\ra]}
\end{equation}
where $q$ is a formal variable.  Then
$$d_\la(e^{\pi i/\ell})=\FPdim(X_\la)$$ is the Frobenius-Perron dimension
for all of the Grothendieck equivalent categories in $\mathfrak{C}(\g,\ell)$.
\end{theorem}
\begin{proof}
First observe that $d_\la(e^{\pi i/\ell})$ is positive since $0<\lan\la+\rho,\check{\al}\ra\leq\lan\la+\rho,\check{\theta}_s\ra=\lan\la+\rho,\theta_s\ra<\ell$ for any positive coroot
$\check{\al}$, so that evaluating any
factor $[\lan\la+\rho,\check{\al}\ra]$ of the numerator of $d_\la(q)$ at $e^{\pi i/\ell}$
gives
$$\sin(\lan\la+\rho,\check{\al}\ra\pi/\ell)>0$$ and similarly for
any factor of the denominator.  Since the Frobenius-Perron dimension
is unique among positive dimension functions, it suffices to show that $d_\la(e^{\pi i/\ell})$ is indeed a dimension function for $\CC(\g,q,\ell)$.  This is achieved first by showing that $d_\la(q)$ is a dimension function for $\Rep(U_q\g)$, and then showing that $d_\la(e^{\pi i/\ell})$ exhibits the requisite $W_\ell$-antisymmetry to derive the result for
$\CC(\g,q,\ell)\in\mathfrak{C}(\g,\ell)$.

Applying the Weyl denominator de-factorization formula to the coroot system with normalized $W$-invariant form $\lan\,\,,\,\ra^\prime=m\lan\,\,,\,\ra$ one obtains
$$d_\la(q)=
\frac{\sum_{w\in W}\ve(w)q^{\lan w(\la+\rho),2\check{\rho}\ra}}{\sum_{w\in W}\ve(w)q^{\lan w(\rho),2\check{\rho}\ra}}$$ where
$\check{\rho}$ is half the sum of the positive coroots.  But this is precisely the
character $\chi_\la$ of the $U_q\g$-module $V_\la$ evaluated at a particular element (see
\cite{Kir}[Section 3] for example).  Since $\chi_\la\chi_\mu=\sum_\nu m_{\la\mu}^\nu\chi_\nu$,
we have that $d_\la(q)$ is a dimension function for $\Rep(U_q\g)$.  To show that $$d_\la(e^{\pi i/\ell})d_\mu(e^{\pi i/\ell})=\sum_\nu N_{\la\mu}^\nu d_\nu(e^{\pi i/\ell})$$
we need only show that $d_{w.\la}(e^{\pi i/\ell})=\ve(w)d_\la(e^{\pi i/\ell})$ for any
$w\in W_\ell$.  It suffices to check that $d_\la(e^{\pi i/\ell})$ is invariant under the dot action of the generator $T_\ell$ (since $\ve(T_\ell)=1$) and anti-invariant under the dot action of the generators $s_i$ (with $\ve(s_i)=-1$). For this we compute:
$$\lan w(s_i.\la+\rho),2\check{\rho}\ra=
\lan w(s_i(\la+\rho)-\rho+\rho),2\check{\rho}\ra=\lan ws_i(\la+\rho),2\check{\rho}\ra$$
and since $\ve(ws_i)=-\ve(w)$ the anti-symmetry with respect to $s_i$ follows by reindexing
the sum.  Similarly for $T_\ell$ we compute (recalling that $T_\ell$ is the translation
by $\ell\check{\theta}_s=\ell\theta_s$ in this case):
\begin{eqnarray*}
&&\lan w(T_\ell.\la+\rho),2\check{\rho}\ra=\lan T_\ell(\la+\rho),w^{-1}(2\check{\rho})\ra=
\lan \la+\rho+\ell\theta_s,w^{-1}(2\check{\rho})\ra=\\
&&\lan w(\la+\rho),2\check{\rho}\ra+\ell\lan w(\theta_s),2\check{\rho}\ra
\end{eqnarray*}
and since $\ell\lan w(\theta_s),2\check{\rho}\ra$ is an \emph{even} multiple of $\ell$
(as $\lan w(\theta_s),\check{\rho}\ra\in\Z$), $q^{\ell\lan w(\theta_s),2\check{\rho}\ra}=1$.
Thus $d_\la(e^{\pi i/\ell})$ is invariant under the dot action of $T_\ell$.  This completes
the proof that $d_\la(e^{\pi i/\ell})=\FPdim(X_\la)$.
\end{proof}
\begin{remark}
The statement of Theorem \ref{fpdim} is false if $m\mid\ell$ and $m\neq 1$.  The proof fails because
the translation $T_\ell$ in those cases is by $\ell\check{\theta}=\frac{\ell}{m}\theta$,
so that $\ell\lan w(\check{\theta}),2\check{\rho}\ra$ \emph{is not} always an even
multiple of $\ell$.
\end{remark}

\subsection{Lie Type $G_2$}
The category $\CC(\mathfrak{g}_2,q,7)$ is the trivial rank $1$ category for any choice of $q$, and the rank of $\CC(\mathfrak{g}_2,q,\ell)$ for $3\nmid\ell$ can be computed via the
generating function $\frac{1}{(1-x)(1-x^2)(1-x^3)}$ from \cite{Survey}.  If we label the
short fundamental weight (corresponding the the $7$-dimensional rep of $\mathfrak{g}_2$) by $\La_1$ and the long fundamental weight by $\La_2$ we obtain
the following formula for $\dim_q$ of the simple object $X_\la$ labeled by weight
$\la=\la_1\La_1+\la_2\La_2$:
\begin{equation*}
\frac{[\la_1+1][3(\la_2+1)][\la_1+3\la_2+4][3(\la_1+\la_2+2)][2\la_1+3\la_2+5][3(\la_1+2\la_2+3)]}{[1][3][4][5][6][9]}.
\end{equation*}
The formula for $\FPdim(X_\la)$ is similar to the above except one must cancel
the factors of $3$ in the $q$-numbers of the form $[3n]$ in the numerator and denominator,
and then evaluate at $e^{\pi i/\ell}$.

For $\ell$ sufficiently large $\CC(\mathfrak{g}_2,q,\ell)$ is not (pseudo-)unitary for
any choice of $q$, but for smaller values of $\ell$ one obtains some unitarity results.  In particular, we have the following:
\begin{theorem}\label{g2thm}
Let $\ell$ be an integer with $\ell> 7$ and $3\nmid\ell$, and $q=e^{\pi i z/\ell}$ with $\gcd(z,\ell)=1$.  Then
\begin{enumerate}
\item[(a)] if $\ell\geq 16$, $\CC(\mathfrak{g}_2,q,\ell)$ is not pseudo-unitary for any choice of $q$,
\item[(b)] the categories $\CC(\mathfrak{g}_2,e^{\pi i 2/11},11)$, $\CC(\mathfrak{g}_2,e^{\pi i 3/13},13)$ and $\CC(\mathfrak{g}_2,e^{\pi i 3/14},14)$ are pseudo-unitary, and
\item[(c)] $\CC(\mathfrak{g}_2,q,8)$ $\CC(\mathfrak{g}_2,q,10)$ and $\CC(\mathfrak{g}_2,q,11)$ are Grothendieck unitarizable.
\end{enumerate}
\end{theorem}
\begin{proof}
For part (a) fix $\ell\geq 16$.  Observe that $X_{\La_1}$ appears as a subobject of $X_{\La_1}^{\ot 2}$, so that if $\CC(\mathfrak{g}_2,q,\ell)$ were pseudo-unitary we would have $\dim_q(X_{\La_1})=\FPdim(X_{\La_1})$ by Lemma \ref{dimlemma}.
Thus for part (a) it is enough to show that $\dim_q\neq\FPdim$ for any choice of $q$, which
can be accomplished by showing $\dim_q(X_{\La_1})<\FPdim(X_{\La_1})$.  Comparing formulas (\ref{fpdim}) and (\ref{qdim}), we see that this amounts
to showing
\begin{equation}\label{g2ineq}
\frac{\sin(7z\pi/\ell)\sin(2z\pi/\ell)\sin(12z\pi/\ell)}{\sin(z\pi/\ell)\sin(4z\pi/\ell)\sin(6z\pi/\ell)}< \frac{\sin(7\pi/\ell)}{\sin(\pi/\ell)}
\end{equation}
for any $1\leq z\leq\ell-1$ with $\gcd(z,\ell)=1$.  By invariance of $\dim_q(X_{\La_1})$ we
may assume further that $1\leq z\leq \frac{\ell-1}{2}$.  For notational convenience let us introduce the new variable $t=z\pi/\ell$ so that $0<t<\pi/2$.  Standard trigonometric identities
yield
$$\dim_q(X_{\La_1})=\frac{\sin(7t)\sin(2t)\sin(12t)}{\sin(t)\sin(4t)\sin(6t)}=\frac{\sin(7t)}{\sin(t)}(4\cos^2(2t)-3).$$
Since $-3\leq 4\cos^2(2t)-3\leq 1$, we consider two cases: (I) $\frac{\sin(7t)}{\sin(t)}<0$ and (II) $\frac{\sin(7t)}{\sin(t)}>0$.

In case (I) we compute the minimum of
$\frac{\sin(7t)}{\sin(t)}$ to be $-\frac{7+14\sqrt{7}}{27}>-1.632$.  So in case (I) we have
that $\dim_q(X_{\La_1})<(-3)(-1.632)<4.9$.  Next notice that $\FPdim(X_{\La_1})$ is
an increasing function of $\ell$, so we have $$5.02<\frac{\sin(7\pi/16)}{\sin(\pi/16)}\leq\FPdim(X_{\La_1})$$ hence in case (I) we see that
(\ref{g2ineq}) holds.

 For case (II) we first observe that (\ref{g2ineq}) holds for $z=1$, since $4\cos^2(2\pi/\ell)-3\neq 1$, reducing the problem to showing
$$\frac{\sin(7z\pi/\ell)}{\sin(z\pi/\ell)}<\frac{\sin(7\pi/\ell)}{\sin(\pi/\ell)}$$ on $2\leq z\leq \frac{\ell-1}{2}$.  One sees that the function $\frac{\sin(7t)}{\sin(t)}$ is decreasing
on the interval $0<t<1/2$, corresponding to $0<z<\frac{\ell}{2\pi}$, so that $\frac{\sin(7z\pi/\ell)}{\sin(z\pi/\ell)}<\frac{\sin(7\pi/\ell)}{\sin(\pi/\ell)}$ certainly holds on the interval $(1,\frac{\ell}{2\pi})$.  Moreover, on $1/2\leq t\leq \pi/2$, we have $\frac{\sin(7t)}{\sin(t)}<2<\frac{\sin(7\pi/\ell)}{\sin(\pi/\ell)}$ so that
 (\ref{g2ineq}) also holds in case (II), completing the proof of (a).

The proof of (b) is simply a computation: one uses formula \ref{qdim} for $\dim_q$ evaluated at
the given value of $q$ to check positivity on the ($5,8$ and $10$) simple objects for the
given categories corresponding to $\ell=11,13$ and $14$.

To prove part (c) we explicitly determine the fusion rules for the given categories
by computing the classical multiplicities $m_{\la\mu}^{\nu}$ (for instance by using
Stembridge's Maple packages \text{coxeter/weyl} \cite{Stem}) and applying formula (\ref{qRacah}).  Then we recognize the fusion rules as those of well-known unitary premodular categories obtained
from the quantum group $U_q\mathfrak{sl}_2$.  In particular, we have:
\begin{enumerate}
\item $\CC(\mathfrak{g}_2,q,8)\greq\CC(\mathfrak{sl}_2,e^{\pi i/3},3)$,

\item $\CC(\mathfrak{g}_2,q,10)\greq\Z(\CC(\mathfrak{sl}_2,e^{\pi i/5},5))\boxtimes \Z(\CC(\mathfrak{sl}_2,e^{\pi i/5},5))$ and
\item  $\CC(\mathfrak{g}_2,q,11)\greq\Z(\CC(\mathfrak{sl}_2,e^{\pi i/11},11))$
\end{enumerate}
where $\Z(\CC(\mathfrak{sl}_2,q,\ell))$ is the subcategory generated by the simple
objects with even labels $X_0,X_2,\cdots$, \emph{i.e.} with categorical dimensions $1=[1],[3],[5],\cdots$,
 and $\boxtimes$ is the tensor product of $\C$-linear categories.
Thus these categories are Grothendieck unitarizable.
\end{proof}

\subsection{Lie type $F_4$}

The category $\CC(\mathfrak{f}_4,q,\ell)$ with $\ell$ odd is trivial for $\ell=13$ and has rank $4,10,21,39...$ for $\ell=15,17,19,21...$ as can be obtained from the generating function $[(1-x)(1-x^2)^2(1-x^3)(1-x^4)]^{-1}$ (see \cite{Survey}).  We label
the four fundamental weights by $\La_i$, $1\leq i\leq 4$ where $\La_1$ corresponds to the $26$-dimensional representation of $\mathfrak{f}_4$.

As in the Lie type $G_2$ case, when $\ell$ is large enough, the categories $\CC(\mathfrak{f}_4,q,\ell)$ are not pseudo-unitary, while for small values of $\ell$ one
does obtain some forms of unitarity.  Specifically we have:
\begin{theorem}\label{f4thm}
Let $\ell$ be an integer with $\ell> 13$ and $2\nmid\ell$, and $q=e^{\pi i z/\ell}$ with $\gcd(z,\ell)=1$.  Then
\begin{enumerate}
\item[(a)] if $\ell\geq 19$, $\CC(\mathfrak{f}_4,q,\ell)$ is not pseudo-unitary for any choice of $q$,
\item[(b)] $\CC(\mathfrak{f}_4,e^{3\pi i/17},17)$ is pseudo-unitary and
\item[(c)] $\CC(\mathfrak{f}_4,q,15)$ and $\CC(\mathfrak{f}_4,q,17)$are Grothendieck unitarizable.
    \end{enumerate}
\end{theorem}
\begin{proof}
The proof of (a) follows the same strategy as in in the Lie type $G_2$ case.  Once again $X_{\La_1}$ appears as a subobject of its own second tensor power, so that pseudo-unitarity would imply $\dim_q(X_{\La_1})=\FPdim(X_{\La_1})$. One reduces to showing
that
\begin{equation}\label{f4ineq}
\frac{\sin(3z\pi/\ell)\sin(8z\pi/\ell)\sin(13z\pi/\ell)\sin(18z\pi/\ell)}{\sin(z\pi/\ell)\sin(4z\pi/\ell)\sin(6z\pi/\ell)\sin(9z\pi/\ell)}<\frac{\sin(8\pi/\ell)\sin(13\pi/\ell)}{\sin(\pi/\ell)\sin(4\pi/\ell)}
\end{equation} for $\ell\geq 19$ and $1\leq z\leq \frac{\ell-1}{2}$.  We omit the details, pausing to note that the relevant trigonometric identity in terms of $t=z\pi/\ell$ is
$$\frac{\sin(18t)\sin(3t)}{\sin(6t)\sin(9t)}=4\cos^2(3t)-3.$$

For part (b) one simply computes $\dim_q$ $10$ simple objects corresponding
to $\ell=17$ and verifies that they are positive.

For part (c) we have the following explicit Grothedieck equivalences with unitary premodular categories:
\begin{enumerate}
\item $\CC(\mathfrak{f}_4,q,15)\greq\Z(\CC(\mathfrak{sl}_2,e^{\pi i/5},5))\boxtimes \Z(\CC(\mathfrak{sl}_2,e^{\pi i/5},5))$
\item $\CC(\mathfrak{f}_4,q,17)\greq\CC(\mathfrak{e}_8,e^{\pi i/34},34)$.
\end{enumerate}
The case $\ell=15$ is fairly simple to verify, while the case $\ell=17$ is handled in the Appendix.
\end{proof}

\subsection{Lie Types $B$ and $C$}

For Lie types $B$ and $C$ we have a somewhat stronger statement.  For these types, we can
combine the classification theorem in \cite{TbWz} and the results in \cite{Rowell1} to show that any braided fusion category that is
Grothendieck equivalent to $\CC(\mathfrak{so}_{2k+1},q,\ell)$ for some $q$ with $\ell$ odd has
the same dimension function for some (possibly different) choice of $q$.  Note that these categories are trivial for $\ell=2k+1$.
We have the following theorem, part (a) of which is a sharpening of \cite{Rowell1}[Theorem 7.1]:
\begin{theorem}\label{bcthm}
Let $\ell$ be an odd integer with $2k+1<\ell$, and $q=e^{\pi i z/\ell}$ with $\gcd(z,\ell)=1$.  Then
\begin{enumerate}
\item[(a)] if $2k+5\leq\ell$, neither $\CC(\mathfrak{so}_{2k+1},q,\ell)$ nor $\CC(\mathfrak{sp}_{\ell-2k-1},q,\ell)$ is pseudo-unitary or Grothendieck unitarizable for any choice of $q$, and
\item[(b)] if $\ell=2k+3$ then both $\CC(\mathfrak{so}_{2k+1},q,\ell)$ and $\CC(\mathfrak{sp}_{\ell-2k-1},q,\ell)$ are Grothendieck unitarizable.
    \end{enumerate}
\end{theorem}
\begin{proof}
First one reduces to considering only $\CC(\mathfrak{so}_{2k+1},q,\ell)$ using a rank-level duality result: $\CC(\mathfrak{so}_{2k+1},q,\ell)\greq \CC(\mathfrak{sp}_{\ell-2k-1},q,\ell)$
established in \cite{Rowell1}[Corollary 6.6].

The proof of (a) relies upon establishing that $\dim_q(X_{\Lambda_1})\neq\FPdim(X_{\Lambda_1})$ where $\Lambda_1$ is the fundamental weight corresponding to the vector representation of $\mathfrak{so}_{2k+1}$.  Since $X_{\Lambda_1}$ is a subobject of $X_{\Lambda_k}\ot X_{\Lambda_k}^*$ (where $X_{\Lambda_k}$ is the object corresponding to the fundamental spin representation of $\mathfrak{so}_{2k+1}$) this is a sufficient condition by Lemma \ref{dimlemma}.  This amounts to showing that:
\begin{equation}\label{bceq}
\frac{\sin((2k+1)z\pi/\ell)\sin(2(2k-1)z\pi/\ell)}{\sin(2z\pi/\ell)\sin((2k-1)z\pi/\ell)}<\frac{\sin((2k+1)\pi/\ell)}{\sin(\pi/\ell)}
\end{equation} for any $1\leq z\leq \ell-1$ with $\gcd(z,\ell)=1$, which was done in \cite{Rowell1}[Lemma 7.2] for $4k+3\leq\ell$.  A referee graciously provided us with the following proof
of Inequality (\ref{bceq}) in the general case:

First notice that the left-hand-side of Ineq. (\ref{bceq}) can be written as 
$\frac{\sin(4kz\pi i/\ell)}{\sin(2z\pi/\ell)}+1$.  Using the identity:

$$\frac{\sin(s\al)}{\sin(\al)}+1=\frac{((-1)^s+1)}{2}+2\sum_{j=0}^{\lfloor\frac{s-1}{2}\rfloor}\cos((s-1-2j)\al)$$ 
(where $\lfloor x\rfloor$ is the greatest
integer function) one sees that Ineq. (\ref{bceq}) is equivalent to:
\begin{equation}\label{bcineq}
\sum_{j=1}^{k}\cos(2(2j-1)z\pi/\ell)<\sum_{j=1}^k\cos(2j\pi/\ell).
\end{equation}

Observe that each side of (\ref{bcineq}) is a sum of $k$ distinct numbers of the form 
$\cos(2s\pi/\ell)$ for some integer $s$ with $1\leq s\leq (2k-1)z$.  Since $\gcd(z,\ell)=1$ and
$2k-1<\ell$ we see that $\ell\nmid s$, hence the largest of these numbers is $\cos(2\pi/\ell)$
followed by $\cos(4\pi/\ell)$ etc. so that the right-hand-side of (\ref{bcineq}) is the
sum of the $k$ largest numbers of this form.  It follows that (\ref{bcineq}) holds provided 
it is not an equality.  For this, first notice that (\ref{bcineq}) 
is an equality if and only if the sets of integers 
$$S_1=\{\pm 1,\pm 2,\ldots,\pm k\}\quad \text{and} \quad S_2=\{\pm z,\pm 3z,\ldots,\pm (2k-1)z\}$$
are equal modulo $\ell$ (since $\cos(x)$ is an even function).

Suppose $S_1=S_2$.
If we imagine the set of integers modulo $\ell$ placed on a circle in the usual way, the subset $S_2$ has the following description: starting at $-(2k-1)z$, the remaining $2k-1$ elements of $S_2$ are obtained by moving counterclockwise in $2k-1$ steps of size $2z$.  The subset $S_1$ consists of all of the inequivalent residues modulo $\ell$ except $0$ and a gap $\I=[k+1,\ell-(k+1)]$ of size $|\I|=\ell-2(k+1)+1\geq 4$.  Consider two cases 1) $2z<|\I|$ and 2) $2z\geq |\I|$.  In case 1) we arrive at an immediate contradiction since after some step of size $2z$ we will arrive at some element of $S_1$ in the gap
$\I=[k+1,\ell-(k+1)]$ (for example, one of $k+2z$ or $(k-1)+2z$ is such an element).  In case 2), consider the three consecutive residues $\ell-k-2z-1,\ell-k-2z-2$ and $\ell-k-2z-3$.  By hypothesis these three numbers are not in the gap, so they must be either in $S_1$ or equal to $0$ and hence at least one of them must be of the form $(2i-1)z$ with $-k+1\leq i<k$ (another possible value is $(2k-1)z$).  Such a number shifted by $2z$ must remain in $S_1$, but each of $\ell-k-1$, $\ell-k-2$ and $\ell-k-3$ is in the gap (as $|\I|\geq 4$), a contradiction.  This proves (a).

For part (b), observe that if $\ell=2k+3$ then $\ell-2k-1=2$, so that the rank-level duality
and the isomorphism $\mathfrak{sp}_{2}\cong\mathfrak{sl}_2$
implies that
$$\CC(\mathfrak{so}_{2k+1},q,\ell)\greq \CC(\mathfrak{sp}_{2},q,\ell)\greq\CC(\mathfrak{sl}_2,q,\ell).$$  It is of course well-known
that $\CC(\mathfrak{sl}_2,q,\ell)$ is a unitary modular category for $q=e^{\pi i/\ell}$ (see e.g. \cite{Wenzl1}).  Thus (b) is established.
\end{proof}

\subsection{Concluding remarks}

\begin{enumerate}
\item A natural question to ask is if Theorems \ref{g2thm}(a) and \ref{f4thm}(a) can be
strengthened, replacing \emph{pseudo-unitary} with \emph{Grothendieck unitarizable} as in Theorem \ref{bcthm}.  Unfortunately a classification for
fusion categories of Lie types $E_N$, $F_4$ or $G_2$ (in the spirit of the type $A-D$ classifications, see \cite{KW} and \cite{TbWz}) does not exist, making such a strengthening
problematic (at least from our approach).
\item  It is desirable to have more conceptual explanation of Theorems \ref{g2thm}, \ref{f4thm} and \ref{bcthm}.  This could potentially be achieved from the observation that for $\CC(\g,q,\ell)$ with $m\mid\ell$ the corresponding categories are monoidally equivalent to the fusion categories of fixed level modules of the (untwisted) affine Kac-Moody algebras $\widehat{\g}$ (see \cite{fink}), while the sets $C_\ell$ described above are related to the labeling sets of fixed level modules of twisted affine Kac-Moody algebras (see \cite{Kac}) \emph{for which no fusion product is known}.  Moreover, A.\ Wasserman pointed out that the fusion category of level $k$
modules for affine Kac-Moody algebras $\widehat{\g}$ are naturally unitary.
\end{enumerate}

\section*{Appendix}

We claim that the categories $\CC(\mathfrak{f}_4,q,17)$ and $\CC(\mathfrak{e}_8,q,34)$ are
Grothendieck equivalent.  We will exhibit an isomorphism of unital based rings between
the corresponding Grothendieck semirings.  Let us label the four
fundamental weights of $U_q\mathfrak{f}_4$ by $\La_i$, $1\leq i\leq 4$ with the label $\La_1$ corresponding to the 26 dimensional representation of $\mathfrak{f}_4$ and the eight
fundamental weights of $U_q\mathfrak{e}_8$ by $\la_i$, $1\leq i\leq 8$ with $\la_8$ corresponding to the 248 dimensional representation of $\mathfrak{e}_8$.  The ten simple objects of $\CC(\mathfrak{f}_4,q,17)$ are labelled by $$\mathcal{B}=\{0,\La_1,\La_2,\La_3,\La_4,2\La_1,2\La_2,\La_1+\La_4,\La_1+\La_2,\La_2+\La_4\},$$
while those of $\CC(\mathfrak{e}_8,q,34)$ are labelled by
$$\mathcal{D}=\{0,\la_8,\la_7,\la_6,2\la_8,\la_1,2\la_1,\la_1+\la_8,\la_2,\la_3\}.$$
Let us use $\mathcal{B}$ and $\mathcal{D}$ to form ordered bases $\{X_i\}_{i=0}^9$ and
$\{Y_i\}_{i=0}^9$ for the Grothendieck semirings $Gr(\CC(\mathfrak{f}_4,q,17))$ and $Gr(\CC(\mathfrak{e}_8,q,34))$ (so that $X_1=X_{\La_1}$ and $Y_1=X_{\la_8}$, etc.).
Using standard techniques (e.g. \cite{Stem}) together with formula (\ref{qRacah}), one computes the fusion matrices $N_{X_1}$ and $N_{Y_1}$ relative to the given ordered bases, and finds that they are identical.
For the reader's convenience we record:
$$N_{X_1}=N_{Y_1}=
\begin{pmatrix} 0&1&0&0&0&0&0&0&0&0
\\1&1&1&0&1&1&0&0&0&0\\0&1&1&1&1&1
&0&1&1&0\\0&0&1&0&0&0&0&1&1&1\\0&1
&1&0&0&0&0&1&0&0\\0&1&1&0&0&1&0&1&1&0
\\0&0&0&0&0&0&0&1&0&1\\0&0&1&1&1&1
&1&1&1&1\\0&0&1&1&0&1&0&1&1&1\\0&0
&0&1&0&0&1&1&1&0\end{pmatrix}
$$

By first writing the fusion rules
for tensoring with $X_1$ as (commutative) polynomials in $X_0,\ldots,X_9$, a Gr\"obner basis computation (with monomials ordered by total degree) explicitly determines the remaining fusion rules.  This proves that the two categories are Grothendieck equivalent.


\begin{thebibliography}{9999}

\bibitem{AndPar} H.\ H.\  Andersen and J.\ Paradowski,
\emph{Fusion categories arising from semisimple Lie algebras}, Comm.
Math. Phys. \textbf{169} (1995), 563--588.

\bibitem{BK} B.\ Bakalov and A.\ Kirillov, Jr.,  Lectures on
Tensor Categories and Modular Functors, University Lecture Series,
vol.\ {\bf 21},  Amer.\ Math.\ Soc., 2001.


\bibitem{ENO} P.\ Etingof, D.\ Nikshych, and V.\ Ostrik, \emph{On fusion categories}, Ann. of Math. (2) \textbf{162} (2005), no. 2, 581-642.

\bibitem{fink} M.\ Finkelberg, \textit{An equivalence of fusion categories.}  Geom. Funct. Anal.  \textbf{6}  (1996),  no. 2, 249--267.

\bibitem{FLKW} M.\ H.\ Freedman, A.\ Kitaev,
    M.\ J.\ Larsen, and Z.\ Wang,
    \emph{Topological quantum computation, Mathematical challenges of the 21st century
    (Los Angeles, CA, 2000)},  Bull.\ Amer.\ Math.\ Soc.\ (N.S.) \textbf{40}  (2003),  no. 1, 31--38.

\bibitem{FLW} M.\ H.\ Freedman, M.\ J.\ Larsen, and Z.\ Wang,
    \emph{The two-eigenvalue problem and density of Jones representation of braid groups,}
    Comm.\ Math.\ Phys.\ \textbf{228} (2002), 177--199.

\bibitem{Kac} V.\ G.\ Kac, Infinite-dimensional Lie algebras. Third edition. Cambridge University Press, Cambridge, 1990.

\bibitem{Kas} C.\ Kassel, Quantum groups. Graduate Texts in Mathematics, 155. Springer-Verlag, New York, 1995.
\bibitem{KW} D.\ Kazhdan and H.\ Wenzl, \textit{Reconstructing monoidal categories.}  I. M. Gelfand Seminar,  111--136, Adv. Soviet Math., 16, Part 2, Amer. Math. Soc., Providence, RI, 1993.

\bibitem{Kir} A.\ Kirillov Jr., \emph{On an inner product in modular categories,} J.\ of Amer.\ Math.\ Soc.\ \textbf{9} (1996) no. 4, 1135--1169.

\bibitem{Ostrik1} V.\ Ostrik, \textit{Module categories, weak Hopf algebras and modular invariants. } Transform. Groups  \textbf{8}  (2003),  no. 2, 177--206.

\bibitem{RR} N.\ Read and E.\ Rezayi,	
\textit{Beyond paired quantum Hall states: Parafermions and incompressible states in the first excited Landau level.}  Phys. Rev. B \textbf{59} (1999), 8084.


\bibitem{Rowell1} E.\ C.\ Rowell, \emph{On a family of non-unitarizable ribbon
categories}, Math. Z. \textbf{250} (2005) no. 4, 745--774.

\bibitem{Survey} E.\ C.\ Rowell \emph{From quantum groups to unitary modular tensor categories} in Contemp.\ Math.\ \textbf{413} (2006), 215--230.

\bibitem{Sawin} S.\ F.\ Sawin, \textit{Quantum groups at roots of unity and modularity.}  J. Knot Theory Ramifications  \textbf{15}  (2006),  no. 10, 1245--1277.

\bibitem{Stem} J.\ R.\ Stembridge, available at \text{http://www.math.lsa.umich.edu/~jrs/maple.html}

\bibitem{TbWz} I.\ Tuba and H.\ Wenzl, \emph{On braided tensor categories of type $BCD$},
  J.\ Reine Angew.\ Math.\ \textbf{581} (2005), 31-69.

\bibitem{Tur} V.\ G.\ Turaev, Quantum Invariants of Knots and 3-Manifolds,
 de Gruyter Studies in Mathematics {\bf 18}, Walter de Gruyter $\&$ Co., Berlin, 1994.

\bibitem{Wenzl1} H.\ G.\ Wenzl, \emph{$C^*$ tensor categories from quantum groups},
J of AMS, \textbf{11} (1998) 261--282.

\end{thebibliography}
\end{document}